\documentclass{article}

\usepackage{amsthm,amsmath,amssymb,graphicx}

\newcommand{\Z}{\mathbb{Z}}

\newcommand{\N}{\mathbb{N}}
\newcommand{\Q}{\mathbb{Q}}

\newtheorem{theoremfoo}{Theorem}[section] 
\newenvironment{theorem}{\pagebreak[1]\begin{theoremfoo}}{\end{theoremfoo}}
\newenvironment{repeatedtheorem}[1]{\vskip 6pt
\noindent
{\bf Theorem #1}\ \em
}{}

\newtheorem{dfntn}[theoremfoo]{Def}
\newenvironment{definition}{\pagebreak[1]\begin{dfntn}\rm}{\end{dfntn}}

\newtheorem{examfoo}[theoremfoo]{Example}
\newenvironment{example}{\pagebreak[1]\begin{examfoo}\rm}{\end{examfoo}}

\newtheorem{commentfoo}[theoremfoo]{Comment}

\newtheorem{corollaryfoo}[theoremfoo]{Corollary}
\newenvironment{corollary}{\pagebreak[1]\begin{corollaryfoo}}{\end{corollaryfoo}}

\newtheorem{lemmafoo}[theoremfoo]{Lemma}
\newenvironment{lemma}{\pagebreak[1]\begin{lemmafoo}}{\end{lemmafoo}}

\newtheorem{notefoo}[theoremfoo]{Note}
\newenvironment{note}{\pagebreak[1]\begin{notefoo}\rm}{\end{notefoo}}

\newtheorem{nttn}[theoremfoo]{Notation}
\newenvironment{notation}{\pagebreak[1]\begin{nttn}\rm}{\end{nttn}}

\begin{document}

\title{An additive version of Ramsey's theorem}
\author{Andy Parrish}
\maketitle

\begin{abstract}
We show that, for every $r, k$, there is an $n = n(r,k)$ so that any
$r$-coloring of the edges of the complete graph on $[n]$
will yield a monochromatic complete subgraph on vertices
$\left\{ a + \sum_{i \in I} d_i \mid I \subseteq [k]\right \}$
for some choice of $a, d_1, \ldots, d_k$.
In particular, there is always a solution to
$x_1 + \ldots + x_\ell = y_1 + \ldots + y_\ell$
whose induced subgraph is monochromatic.
\end{abstract}

\section{Introduction}
Given a set $X$ and a number $r$, an $r$-coloring of $X$ is any map
$\chi:X \rightarrow [r]$, where $[r] = \{1, \ldots, r\}$ is the set
of colors.

Ramsey's celebrated theorem \cite{ramsey} states that, given $r, k$,
there is an $R=R(r,k)$ so that any $r$-coloring of the edges of the
complete graph on $R$ vertices contains a monochromatic complete graph
on $k$ vertices. In addition to being an important result in itself,
Ramsey is the namesake of a large field of research into Ramsey Theory,
which more generally tells when a coloring of a large structure is
guaranteed to have large monochromatic substructures. There are many
great resources on Ramsey Theory, but the main source is due to Graham,
Rothschild, and Spencer \cite{GRS}.

The first result in Ramsey theory was actually proved by Hilbert in 1892,
predating Ramsey's theorem (1930) by several decades. Given natural numbers
$a, d_1, \ldots, d_k$, define
\[
H(a; d_1, \ldots, d_k) =
  \left\{ a + \sum_{i \in I} d_i \mid I \subseteq [k] \right\}.
\]
We call such a set $H(a; d_1, \ldots, d_k)$ a Hilbert cube of dimension $k$.
Hilbert proved \cite{hilbert} that, given $r, k$ natural numbers, there is a
number $H = H(r,k)$ so that any $r$-coloring of $[H]$ contains a
monochromatic Hilbert cube of dimension $k$.

It was further shown that finite-colorings of natural numbers would
always contain monochromatic solutions to $x+y=z$ (Schur \cite{schur}),
as well as long monochromatic arithmetic progressions (van der Waerden
\cite{vdw}). The holy grail of results of this type is Rado's theorem
\cite{rado}, which characterizes which systems of linear equations have
monochromatic solutions under every finite-coloring of the naturals.
Those which do are called {\it partition-regular}.

These results are philosophically related to Ramsey's theorem, but
the graph theoretic and additive sides of Ramsey theory are largely
distinct fields. In recent years, however, Deuber, Gunderson, Hindman,
and Strauss proved a connecting result \cite{gunder1} --- for any $m$,
any sufficiently large graph either contains a $K_{m,m}$, or else it
has an independent set with a prescribed additive structure. Later,
Gunderson, Leader, Pr\"{o}mel, and R\"{o}dl showed \cite{gunder2} that
for any $m, k$, large graphs must either contain a $K_m$ or there must
be an arithmetic progression of length $k$ which is an independent set.
There have been further results in this area, but none give a purely
additive result.

Our goal is to find some additive property $\mathcal{P}$ for which we
can guarantee that every finite edge-coloring of the complete graph
on $[n]$ will contain a set of vertices with property $\mathcal{P}$
whose induced subgraph is monochromatic. In this paper, we consider
properties where $X$ has property $\mathcal{P}$ if $X$ satisfies a
particular system of linear equations.

We always demand solutions by distinct values, so that the monochromatic
subgraphs are non-trivial. For example, if we solve $x+y=z$ by
$x=y=3, z=6$, the corresponding graph has only a single edge, $\{3,6\}$.
The induced graph has no choice but to be monochromatic.

Formally, given a matrix $B$ and a number of colors $r$, we would like
to know whether there is an $n = n(r,B)$ so that any $r$-edge-coloring
of the complete graph on $[n]$ gives a vector
$\vec{x} = (x_1, \ldots, x_k)$ of distinct entries so that the values
$\{x_1, \ldots, x_k\}$ are monochromatic, and $B \vec{x} = \vec{0}$.

We consider this problem for many systems of equations known to be
partition-regular. In Section~\ref{se:negatives}, we give several
negative results. In Section~\ref{se:twocolors} we give an initial
positive result: there is an $n$ so that, any 2-coloring of the edges
of the complete graph on $[n]$ gives a monochromatic 2-dimensional
Hilbert cube. In Section~\ref{se:trees}, we prove a lemma about
coloring $k$-ary trees which may be interesting on its own.
In Section~\ref{se:full} we extend our initial result to any
number of colors and to Hilbert cubes of any size. We believe these
are the first explicit\footnote
{In personal communication, David Conlon noted that our main result in fact
follows from the Graham-Rothschild theorem on $n$-parameter sets \cite{GR}.
The proof in this paper is preferred as it is more easily expanded to
similar results.}
positive results in this direction.

\section{Negative results}\label{se:negatives}
There are many families of equations for which monochromatic solutions
can be easily avoided in this graph setting.

\subsection{Arithmetic progressions}

Van der Waerden's theorem \cite{vdw} tells us that any finite-coloring of
the naturals have arbitrarily long monochromatic arithmetic progressions.
What can we say when coloring pairs of naturals? An arithmetic progression
of length 3 is given by $a, a+d, a+2d$. We notice that the triple contains
two differences: $d$ and $2d$. This observation allows us to 2-color
the complete graph on the naturals without a monochromatic 3-AP.

The coloring is simple. For a pair $\{x, y\}$, write $|x-y| = 2^p q$
where $p, q$ are integers and $q$ is odd. If $p$ is even, color $\{x,y\}$
red. Otherwise, color it blue.

Now let $a, a+d, a+2d$ be a 3-AP. Write $d = 2^p q$. Then we see
$2d = 2^{p+1} q$, so the edges $\{a,a+d\}$ and $\{a,a+2d\}$ have
different colors.

This coloring avoids 3-APs, so we certainly cannot hope for anything longer.

\subsection{Schur's equation and generalizations}

Schur's theorem \cite{schur} states that any finite-coloring of the
naturals has a mono-chromatic solution to $x+y=z$. Additionally, it follows from
Folkman's theorem that there is a monochromatic solution to
$x_1 + \ldots + x_k = z$ for arbitrary $k$.

More generally, we consider equations of the form
\begin{equation}\label{eq:schur}
a_1 x_1 + \ldots + a_k x_k = b z
\end{equation}
with $a_1, \ldots, a_k \ge b > 0$.

We note that any solution to Equation~\ref{eq:schur} has $x_i \le z$ for
$i = 1, \ldots, k$. Using two colors, we can ensure that every graph
induced by a solution to an equation of this form in the natural numbers
contains both colors. We first show how to avoid $x+y=z$ as motivation for
the approach, and then handle the general case.

If $x+y=z$ then either $x$ or $y$ is smaller than their average,
$\frac{1}{2}z$, and the other must be larger than their average.
Thus, given a pair $\{u,v\}$ with $u < v$, we color it red
if $u \le \frac{1}{2}v$, and blue if $u > \frac{1}{2}v$.
Now we see that whenever $x+y=z$, the largest of the three numbers must
be $z$. Either $x$ or $y$ is smaller than $\frac{1}{2}z$, and the other
is larger, so the pairs $\{x,z\}$ and $\{y,z\}$ have different colors.
(Recall that we are only interested in solutions by distinct numbers).

In Equation~\ref{eq:schur}, a similar logic applies. We see that
$a_i x_i \le bz$. Since $a_i \ge b > 0$, we get $x_i \le z$ as before.
Let $M = a_1 + \ldots + a_k$. Divide both sides of the equation by $M$
to get
\[
\frac{a_1}{M} x_1 + \ldots + \frac{a_k}{M} x_k = \frac{b}{M} z.
\]
This says that the weighted average of the $x_i$'s is $\frac{b}{M} z$.
Again, one of the $x_i$'s must be smaller than their average, and another
must be larger. Thus, when $u < v$, we should color $\{u,v\}$ red if
$u \le \frac{b}{M}v$, and blue otherwise. We immediately see that one
of the pairs $\{x_i, z\}$ must be red and another must be blue.

\vspace{1pc}
\noindent
{\bf Remark:}

The argument given above is really a greedy coloring. At step $t$, color
the pairs $\{1,t\}, \ldots, \{t-1,t\}$ in a way that handles those solutions
to Equation~\ref{eq:schur} with largest element $t$. Since we can manage all
these solutions at once, we avoid all monochromatic solutions. The incredible
thing to notice here is that this coloring is much stronger than needed. If
$x_1, \ldots, x_k, z$ satisfy Equation~\ref{eq:schur}, then the star
connecting $z$ to all of the $x_i$'s is not even monochromatic. Forget about
the clique! The strength of this technique suggests that we may be able to
handle a larger family of equations.

On the other hand, this technique relies heavily on the numbers being
positive. If we change the underlying set to $\Z$ or $\Z_p$, the
approach falls apart.

\subsection{Three variables, six colors}

As with many problems in Ramsey theory, we may consider our conjecture as a
hypergraph coloring problem. The vertex set is all pairs we are considering
(be they pairs in $[n], \N, \Z, \Z_n$, etc). For each solution
$(x_1, \ldots, x_k)$ to $b_1 x_1 + \ldots + b_k x_k = 0$, there is a
hyperedge containing all pairs of the $x_i$'s. If we properly
color this $\binom{k}{2}$-uniform hypergraph (avoiding monochromatic
hyperedges), then there are no monochromatic solutions to the equation.
Thus we may apply theorems about hypergraph coloring.

For an equation in three variables, this hypergraph is {\it simple} --- any
two pairs are either disjoint (and have no hyperedges in common), or have
the form $\{x,y\}, \{x,z\}$, leaving only $\{y,z\}$ to form a hyperedge.

Fix $a,b,c$, and consider the hypergraph formed as above by the equation
\begin{equation}\label{eq:three}
ax+by+cz=0.
\end{equation}

Consider a pair $\{u,v\}$. How many hyperedges can it be contained in? Well,
there are 6 different ways of assigning the values $u$ and $v$ to the
variables in Equation~\ref{eq:three}:
\[
\begin{array}{rcl}
au+bv+cz=0 & \implies & z = -\frac{au+bv}{c} \\
av+bu+cz=0 & \implies & z = -\frac{av+bu}{c} \\
au+by+cv=0 & \implies & y = -\frac{au+cv}{b} \\
av+by+cu=0 & \implies & y = -\frac{av+cu}{b} \\
ax+bu+cv=0 & \implies & x = -\frac{bu+cv}{a} \\
ax+bv+cu=0 & \implies & x = -\frac{bv+cu}{a} \\
\end{array}
\]
Thus we see that, so long as the numbers $a, b, c$ are all invertible, each
pair $\{u,v\}$ is contained in at most 6 hyperedges. In particular, if we
are in $\Z, \Q$, or $\Z_p$ for a prime $p$, then the degree is at most 6.
The hypergraph version of Brooks' theorem \cite{brooks} applies.

\begin{theorem}
If $H$ is a hypergraph with maximum degree $\Delta$, then
$\chi(H) \le \Delta$ except in these cases:
\begin{enumerate}
\item $\Delta = 1$,
\item $\Delta = 2$ and $H$ contains an odd cycle (an ordinary graph),
\item $H$ contains a $K_{\Delta}$ (an ordinary graph).
\end{enumerate}
\end{theorem}

Since all of these cases are irrelevant --- ours is a 3-uniform hypergraph,
and we don't have any illusions that we can 1-color it --- this tells us we
can properly 6-color our hypergraph. By construction, this avoids
monochromatic solutions to Equation~\ref{eq:three}.

Moreover, if for example $a=b$, then the six solutions reduce to three
distinguishable ones, meaning 3 colors is enough.

\begin{note}
We avoided considering solutions over $\Z_n$ with $n$ composite and
$a, b, c$ not necessarily invertible. Taken to extremes, this case is
quite degenerate. Consider, for example, $n=2^r$, and $a=b=2^{r-1}$.
Any collection of even numbers then solves Equation~\ref{eq:three}.
The problem of finding a solution which induces a monochromatic subgraph
now reduces to the multicolor Ramsey's theorem for triangles.
\end{note}

\section{Two colors, two dimensions}\label{se:twocolors}

We will eventually prove the following result:

\begin{theorem}\label{th:hilbert}
For all $r,k$, there is a number $n = n(r,k)$ so that any $r$-coloring
of the edges of the complete graph on $[n]$ gives a Hilbert cube
$H = H(a; d_1, \ldots, d_k)$ so that all edges in $H$ are the same color,
and the $2^k$ elements of $H$ are distinct.
\end{theorem}

We first prove the theorem for $r=k=2$. Note that a 2-dimensional
Hilbert cube is four numbers of the form $a, a+b, a+c, a+b+c$.
We will then extend those ideas to any number of colors, and then
to Hilbert cubes of any dimension.

The proof will rely on the Gallai-Witt theorem \cite{gallai-witt}, and a
consequence of Rado's theorem \cite{rado}, both of which we state here.

\begin{theorem}[Gallai-Witt]
For all $r, k$, there exists $GW=GW(r,k)$ so that any $r$-coloring of
$[GW] \times [GW]$ gives numbers $x, y, d$ with the property that
\[
\left\{(x+id, y+jd) \mid i, j = 0, \ldots, k-1 \right\}
\]
are all the same color.
\end{theorem}

\begin{theorem}[Corollary to Rado]\label{th:helper}
There is a number $T$ so that any 2-coloring of $[T]$ gives distinct
numbers $i, j, i+j, j-i$, all the same color.
\end{theorem}

Note: Rado's theorem gives conditions for a system of linear equations
to have monochromatic solutions by distinct numbers. It is a simple
exercise to check that the above satisfies them.


\vspace{1pc}

\begin{proof}[Proof of Theorem~\ref{th:hilbert} when $r=k=2$]
Define $S = GW(T+1, 2)$, where $T$ comes from Theorem~\ref{th:helper}.
We will show that $n = 2S$ suffices.

Fix an $2$-coloring $\chi : \binom{[n]}{2} \rightarrow [2]$. We would
like to find a solution to $w+x=y+z$ which forms a monochromatic clique.
We view $\chi$ as a coloring of the upper half of the lattice
$[n] \times [n]$ --- for $x < y$, the color of $(x,y)$ is $\chi(\{x,y\})$.

Consider the top left quadrant of our grid:
$\{1, \ldots, S\} \times \{S+1, \ldots, 2S\}$.
Define $\chi' : [S] \times [S] \rightarrow [2]$ by
\[
\chi'(a, b) = \chi(a, S+b).
\]
Since $S = GW(T+1, 2)$, and $\chi'$ is a $2$-coloring of $[S] \times [S]$,
we may apply Gallai-Witt to find $x, y, d$ so that all points of the form
\[
\{ (x+id, y+jd) \mid i, j=0, \ldots, T \}
\]
are the same color, say red, under $\chi$. We will consider each subsquare
of this large grid.

For now, consider a red square given by the points
\[
(a,b) \quad (a+h, b) \quad (a, b+h) \quad (a+h, b+h).
\]
We may rewrite the underlying numbers as $a, a+h, a+(b-a), a+h+(b-a)$
to see they form a Hilbert cube of dimension 2.

There are six edges in the graph on these four numbers, and we know
that four of them are red. Thus, we only need to consider the edges
$\{a, a+h\}$ and $\{b, b+h\}$. If these are both red (and the four
values are distinct), then we have the desired monochromatic 4-clique.
Thus, either we have our goal, or every red square gives us two edges
which cannot both be red.


Well, we have a great many red squares. Each has corner $(x+id, y+jd)$
and side-length $\ell d$, for every choice of $i, j, \ell$ with
$i, j, i+\ell, j+\ell$ all in $\{0, \ldots, S\}$. The four underlying
numbers are all distinct by the choice of our initial grid
$\{1, \ldots, S\} \times \{S+1, \ldots, 2S\}$. The ``final'' edges of
this square are $\{x+id, x+(i+\ell)d\}$ and $\{y+jd, y+(j+\ell)d\}$,
so these two cannot both be red without reaching our goal.

All of our red squares will give us many interacting conditions, which we
record in a graph. Let $G = (A \cup B, E)$ be a bipartite graph, where
$A = B = \binom{ \{0, \ldots, T\} }{2}$. We say $\{a, a'\} \sim \{b, b'\}$
if $\{x+ad, x+a'd\}$ and $\{y+bd, y+b'd\}$ are the final edges of some
red square. There is an induced 2-coloring of both $A$ and $B$ --- namely
\[
\chi_A(\{i,j\}) = \chi(x+id, x+jd),
\]
\[
\chi_B(\{i,j\}) = \chi(y+id, y+jd)
\]

\begin{figure}[t]
  \centering
  \includegraphics[scale=0.5]{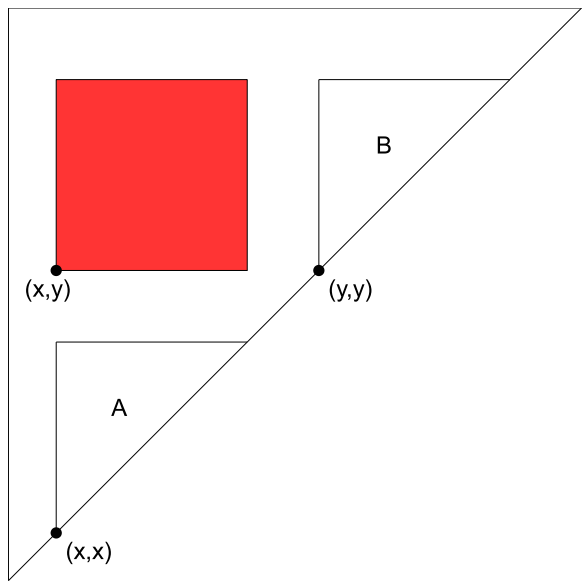}
  \caption{A large red grid, and the corresponding sets $A$ and $B$}
  \label{fig:firstgrid}
\end{figure}

We see immediately that $\{i, i+\ell\} \sim \{j, j+\ell\}$ so long as
those numbers are all in $\{0, \ldots, T\}$. This means that each pair
in $A$ with difference $\ell$ is connected to every pair in $B$ with
that difference. This means that if one pair in $A$ is red, all pairs
in $B$ with that difference must be blue (and vice versa). In fact,
this is the entire structure of $G$.

Write $A = A_1 \cup A_2 \cup \ldots \cup A_T$, where $A_{\ell}$
contains all pairs in $A$ of the form $\{i, i+\ell\}$. We now 2-color
$[T]$, the index set of the $A_{\ell}$'s. Say $\phi(\ell) = \hbox{red}$
if {\it any} pair in $A_{\ell}$ is red. Otherwise,
$\phi(\ell) = \hbox{blue}$, meaning that $A_{\ell}$ is entirely blue.
Since $\phi$ is a 2-coloring of $[T]$, Theorem~\ref{th:helper} tells
us there are distinct numbers $i, j, i+j, j-i$ which are monochromatic.

{\bf Case 1:} The numbers are red. This means each set
$A_i, A_j, A_{i+j}, A_{j-i}$ contains a red pair.
Therefore the corresponding sets in $B$, what we should call
$B_i, B_j, B_{i+j}, B_{j-i}$, are all entirely blue. The proof continues
as in case 2 below, but with all $A$'s changed to $B$'s, and all $x$'s
changed to $y$'s.

{\bf Case 2:} The numbers are blue, so all pairs in
$A_i, A_j, A_{i+j}, A_{j-i}$ are blue. We list the relevant blue pairs:
\[
\begin{array}{rl}
\hbox{In } A_i:     & \{0, i\}, \{j, i+j\} \\
\hbox{In } A_j:     & \{0, j\}, \{i, i+j\} \\
\hbox{In } A_{i+j}: & \{0, i+j\} \\
\hbox{In } A_{j-i}: & \{i, i+(j-i)\} = \{i, j\}.
\end{array}
\]
Taken together, we see that $0, i, j, i+j$ form a blue $K_4$ under $\chi_A$.
Recalling the relationship between $\chi$ and $\chi_A$, this gives us a
blue $K_4$ under $\chi$ with vertices $x, x+id, x+jd, x+(i+j)d$. This is
the desired 2-dimensional Hilbert cube.
\end{proof}

\section{Coloring $k$-ary trees}\label{se:trees}

In order to achieve Theorem~\ref{th:hilbert} for any number of colors,
we will first require a Ramsey-type theorem for $k$-ary trees.

\begin{notation}
We use $[k]^*$ to denote all finite sequences (strings) of elements of
$[k] = \{1, \ldots, k\}$. If $s, t \in [k]^*$, we use
$s \cdot t$ to denote concatenation --- all characters of $s$ followed
by all characters of $t$.
\end{notation}

\begin{definition}
A perfect $k$-ary tree $T_n^{(k)}$ of height $n$ is the collection of nodes
\[
T_n^{(k)} = \{ s \in \{1, \ldots, k\}^j \mid 0 \le j \le n \}.
\]
We say $\lambda$, the empty string, is the root of the tree.
A node $s$ has $k$ children, $s \cdot 1, \ldots, s \cdot k$.
The child $s \cdot i$ together with all of its descendants forms
the $i^{\hbox{th}}$ subtree of $s$, rooted at $s \cdot i$.
We see that $s$ has $k$ subtrees in all.
The $j^{\hbox{th}}$ level of $T_n^{(k)}$ consists of all those strings
of length exactly $j$. The substrings of $s$ are called the {\it ancestors}
of $s$. The nodes at level $n$ are called leaves. If $s$ is a substring
of $t$, we say that the {\it path} from $s$ to $t$ is the set of nodes $r$
which are both superstrings of $s$ and substrings of $t$ (including $s$ and
$t$). The length of the path is the difference in lengths of $s$ and $t$.
\end{definition}

Since we are only interested in perfect $k$-ary trees in this paper, we
will usually refer to them simply as ``$k$-ary trees'', or ``trees''
if $k$ is implied.

Next, we define what it means to embed one $k$-ary tree into another.

\begin{definition}
Let $T, R$ be two $k$-ary trees. An embedding of $T$ into $R$ is a map
$\varphi$ from the nodes of $T$ into the nodes of $R$ with the
following properties:
\begin{enumerate}
\item There is an increasing function $x$ from levels of $T$ to levels of $R$
  so that, if $t$ is on the $i^{\hbox{th}}$ level of $T$, then $\varphi(t)$
  is on the $x(i)^{\hbox{th}}$ level of $R$.
\item If $s, t$ are nodes in $T$ , and $s$ is contained in the
  $i^{\hbox{th}}$ subtree of $t$, then $\varphi(s)$ is contained in the
  $i^{\hbox{th}}$ subtree of $\varphi(t)$.
\end{enumerate}
\end{definition}

We now state the goal of this section:

\begin{lemma}\label{le:embedding}
For every $k, c, n$, there is a number $E = E(k,c,n)$ so that
every $c$-coloring of the $k$-ary tree of height $E$ yields
a monochromatic embedding of the $k$-ary tree of height $n$.
\end{lemma}

We say that a coloring is {\it $n$-balanced} if the conclusion holds.

\begin{example}
A coloring $\chi$ of the a tree $T$ is 1-balanced if there is some node
$r$, and strings $s_1, \ldots, s_k \in [k]^j$ for some $j$, so that
\[
r, r \cdot 1 \cdot s_1, \ldots, r \cdot k \cdot s_k
\]
are all the same color.
This corresponds to the embedding $\varphi$ of $T_1^{(k)}$ into $T$
given by $\varphi(\lambda) = r, \varphi(i) = r \cdot i \cdot s_i$.
\end{example}

\begin{figure}[t]
  \centering
  \includegraphics[scale=0.5]{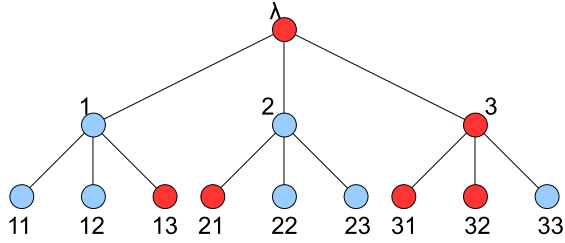}
  \caption{This 3-ary tree is 1-balanced. The nodes $\lambda$, 13, 21,
    and 31 are all red}
  \label{fig:tree}
\end{figure}

We prove Lemma~\ref{le:embedding} by first finding
$f(k, c) = E(k, c, 1)$, and repeatedly applying that result.

\begin{lemma}\label{le:star}
There is a function $f(k, c)$ so that, if $n \ge f(k,c)$, then every
$c$-coloring of the perfect $k$-ary tree of depth $n$ is 1-balanced.
\end{lemma}

We take the proof slowly to delicately handle each part.

\begin{proof}
When $c=1$ the nodes $\lambda, 1, 2, \ldots, k$ must all be the same
color, so they 1-balance the tree. Thus $f(k, 1) = 1$.

Consider the case $k=1$, so that each level has a unique
node. By the pigeonhole principle, the $c+1$ nodes within levels
$0, 1, \ldots, c$ must contain two with the same color. Thus
$f(1,c) = c$.

We begin the same way for $k = 2$. Since we know $f(2, 1)=1$, we work by
induction on $c$. We will show $f(2, c+1) \le (c+1)(1+f(2,c))$.
Call this number $n$.

Let $\chi:T_n^{(2)} \rightarrow [c+1]$ be a $(c+1)$-coloring of the
tree of height $n$. Consider the path from the root to the node $1^n$.
The path contains $n+1$ nodes, so some color is represented at least
\[
\left \lceil \frac{n+1}{c+1} \right \rceil = 2+f(2,c)
\]
times. Call the repeated color ``red.'' Call the levels of these red
nodes $j(-1), j(0), j(1), \ldots, j(f(2,c))$.
Since we looked down the path of all 1s, the corresponding nodes are
\[
\begin{array}{rcl}
r          &   =    & 1^{j(-1)}         \\
s_0        &   =    & 1^{j(0)}          \\
           & \vdots &                   \\
s_{f(2,c)} &   =    &    1^{j(f(2,c))}. \\
\end{array}
\]
Consider $r$ along with any of the other red nodes, $s_i$.
These may be part of a balancing triple --- if any descendent $t$ of
$r \cdot 2$ on level $j(i)$ is also red, then $r, s_i, t$ balance the
tree. Thus, if the tree is to be unbalanced, all of the levels
$j(0), \ldots, j(f(2,c))$ within the second subtree of $r$ must be
entirely non-red. We will now use the definition of $f(2,c)$ to show
that this tree is in fact 1-balanced by these non-red nodes.

Consider the map from the nodes of $T=T_{f(2,c)}^{(2)}$ into our tree
given by
\[
\varphi(\lambda) =
  r \cdot 2^{j(0) - j(-1)} = 1^{j(-1)} \cdot 2^{j(0) - j(-1)},
\]
\[
\varphi(a_1 \ldots a_{\ell-1} a_\ell) =
  \varphi(a_1 \ldots a_{\ell-1}) \cdot (a_\ell)^{j(\ell) - j(\ell-1)}.
\]
We make the following observations:
\begin{enumerate}
\item Nothing in the image of $\varphi$ is red (unless the
  coloring is 1-balanced).
\item All nodes on level $i$ of $T$ are mapped to level $j(i)$ of our tree.
\item If $t$ is contained in the $i^{\hbox{th}}$ subtree of $s$, then
  $\varphi(t)$ is contained in the $i^{\hbox{th}}$ subtree of $\varphi(s)$.
\end{enumerate}
We color $T$ by $\chi^*(s) = \chi(\varphi(s))$, the
coloring induced by $\varphi$. Observation 1 tells us that $\chi^*$
is actually a $c$-coloring.
By the definition of $f(2,c)$, we know that there are some nodes
$w, w \cdot 1 \cdot s_1, w \cdot 2 \cdot s_2$ (with the latter two
on the same level) which are all the same color under $\chi^*$.
Thus we see that $\varphi(w), \varphi(w \cdot 1 \cdot s_1),$ and
$\varphi(w \cdot 2 \cdot s_2)$ must be the same color under $\chi$.
By observations 2 and 3, these nodes 1-balance the original tree.

\vspace{1pc}
\noindent
Finally, for $k \ge 3$, we follow a very similar idea. We will show
\[
f(k, c+1) \le (c+1)(1+(k-1)f(k,c)) = n.
\]

Let $\chi:T_n^{(k)} \rightarrow [c+1]$ be a $(c+1)$-coloring of the $k$-ary
tree of height $n$. Consider the path from the root to the node $1^n$.
The path contains $n+1$ nodes, so some color is represented at least
\[
\left \lceil \frac{n+1}{c+1} \right \rceil = 2+(k-1)f(k,c)
\]
times. Call the repeated color ``red.'' Call the levels of these red
nodes $j(-1), j(0), j(1), \ldots, j((k-1)f(k,c))$.
Since we looked down the path of all 1s, the corresponding nodes are
\[
\begin{array}{rcl}
r               &   =    & 1^{j(-1)}           \\
s_0             &   =    & 1^{j(0)}            \\
                & \vdots &                     \\
s_{(k-1)f(k,c)} &   =    & 1^{j((k-1)f(k,c))}. \\
\end{array}
\]
Consider $j(-1)$ along with any of the other red nodes, $s_i$.
These may be part of a balancing set --- if every subtree has a red node on
the same level, then the coloring is 1-balanced.
Thus, if the tree is to be unbalanced, each of the levels
$j(0), \ldots, j((k-1)f(k,c))$ must be entirely non-red in at least one
of the $k-1$ subtrees of $r$. By the pigeonhole principle,
some subtree of $r$, say the $p^{\hbox{th}}$ subtree, must be colored
such that at least $1+f(k,c)$ of the levels $j(0), \ldots, j((k-1)f(k,c))$
are entirely non-red.
Label these levels $x(0), x(1), \ldots, x(f(k,c))$. We will now use
the definition of $f(k,c)$ to show that this tree is in fact balanced
by these non-red nodes.

Consider the map from $T=T_{f(k,c)}^{(k)}$ into our tree given by
\[
\varphi(\lambda) =
  r \cdot p^{x(0) - x(-1)} = 1^{x(-1)} \cdot p^{x(0) - x(-1)},
\]
\[
\varphi(a_1 \ldots a_{\ell-1} a_\ell) =
  \varphi(a_1 \ldots a_{\ell-1}) \cdot (a_\ell)^{x(\ell) - x(\ell-1)}
\]
We now make the same observations as before:
\begin{enumerate}
\item Nothing in the image of $\varphi$ is red (unless the
  coloring is 1-balanced).
\item All nodes on level $i$ of $T$ are mapped to level $x(i)$ of our tree.
\item If $t$ is contained in the $i^{\hbox{th}}$ subtree of $s$, then
  $\varphi(t)$ is contained in the $i^{\hbox{th}}$ subtree of $\varphi(s)$.
\end{enumerate}
We color $T$ by $\chi^*(s) = \chi(\varphi(s))$, the coloring induced
by $\varphi$. Observation 1 tells us that $\chi^*$
is actually a $c$-coloring.
By the definition of $f(k,c)$, we know that there are some nodes
$w, w \cdot 1 \cdot s_1, \ldots, w \cdot k \cdot s_k$ (with the last $k$
on the same level) which are all the same color under $\chi^*$.
Thus we see that $\varphi(w), \varphi(w \cdot 1 \cdot s_1), \ldots,
\varphi(w \cdot k \cdot s_k)$ must be the same color under $\chi$.
By observations 2 and 3, these nodes 1-balance the original tree.
\end{proof}

The solution to the recurrence bounding $f(k,c)$ for $c,k \ge 2$ gives
\[
f(k,c) = \lfloor e^{1/(k-1)} (k-1)^{c-1} c! \rfloor,
\]
though the true value may be lower.

We may now prove the existence of $E(k,c,n)$.

\vspace{1pc}
\begin{proof}[Proof of Lemma~\ref{le:embedding}]
We only show the result for $n = 2^{\ell}-1$, since this implies all
smaller values. The case $\ell = 1$ is Lemma~\ref{le:star}.

Suppose $E(k,c',2^{\ell}-1)$ is known for all values $c'$. We will
find a bound for $E(k,c,2^{\ell+1}-1)$.

Let $\chi_0 = \chi$ be a $c$-coloring of a large $k$-ary tree. We ignore
the specific height for now, but will determine a bound at the end.

By induction, $\chi_0$ gives a monochromatic embedding $\varphi$
of a $k$-ary tree of height $2^{\ell}-1$ into our large tree, hitting only
levels up to $E(k,c,2^{\ell}-1)$. Call the image $T_{\lambda}$,
and its color $\psi(\lambda)$. $T_{\lambda}$ has $k^{2^\ell - 1}$ leaves,
and each has $k$ subtrees, so we have a total of $Y:=k^{2^\ell}$ subtrees
coming off of $T_{\lambda}$. The roots of these subtrees are given by
\[
\{ v_s = \varphi(t) \cdot i \mid t \in [k]^{2^\ell-1}, s = t \cdot i \}
\]
To each $t \in [k]^*$ we associate a map
$\chi_1(t)$ from $[k]^{2^\ell}$ to $[c]$, given by
\[
\chi_1(t)(s) = \chi(v_s \cdot t).
\]
Note that there are ``only'' $c^Y$ such maps $\chi_1(t)$.
Since each $t$ is mapped to one of $c^Y$ elements, we treat $\chi_1$
as a $c^Y$-coloring of a $k$-ary tree. We think of the $Y$ subtrees of
$T_{\lambda}$ as one tree, where each node is given a list of $Y$ colors,
coming from the vertex in $t$'s position in each of these subtrees.

Because $\chi_1$ is a $c^Y$-coloring of a $k$-ary tree, we know that
there is an embedded $k$-ary tree contained within levels
$0, 1, \ldots, E(k,c^Y,2^{\ell}-1)$ which is monochromatic under
$\chi_1$. Looking back to $\chi$, this means we really have $Y$ trees,
each monochromatic. We label these trees by $T_s$ for $s \in [k]^{2^\ell}$,
based on their connection to $T_{\lambda}$.
Note that each $T_s$ is in the same position relative to $v_s$.
In particular, all the nodes at level $i$ of some $T_s$ are on the same level
in the original tree (regardless of the choice of $s$).
This means that, if all these trees were red, taking them all together
with $T_\lambda$ would give us our monochromatic embedded tree of
height $2^{\ell+1}-1$. Would that we were so lucky.

Instead, all we know is that, for each $s$, the entire tree $T_s$
has some color; call it $\psi(s)$.

We now have $k^{2^\ell}$ trees, each with $k^{2^{\ell}-1}$ leaves,
which in turn each have $k$ subtrees.
Altogether, that gives us $Y^2 = k^{2 \cdot 2^{\ell}}$ subtrees.
We repeat the above argument to get a $c^{Y^2}$-coloring, $\chi_2$,
of the original $k$-ary tree, corresponding to the colors in the subtrees.
We again find a large embedded tree which is monochromatic under $\chi_2$,
and it again corresponds to many trees $T_s$,
each with color $\psi(s)$ under $\chi$. But this time
\[
s \in [k]^{2 \cdot 2^{\ell}}
  = \left( [k]^{2^\ell} \right)^2.
\]

We repeat this process, reaching $\chi_{f(Y, c)}$. The monochromatic
trees here are $T_s$ with color $\psi(s)$, where
\[
s \in [k]^{f(Y,c) \cdot 2^{\ell}}
  = \left( [k]^{2^\ell} \right)^{f(Y,c)}.
\]
We consider the trees $\{T_s\}$ to be the nodes of a large $Y$-ary tree,
colored by $\psi$. Since $\psi$ is a $c$-coloring, and this tree has
height $f(Y,c)$, we get some monochromatic embedded subtree of height 1.
Expanding the nodes as the full trees they are, and observing the
relative structure, we find that these trees form a monochromatic
embedding of a $k$-ary tree of height $2^{\ell+1}-1$, as desired.

In all, we needed to go a depth of
\[
E(k,c,2^\ell-1) + E(k,c^Y,2^\ell-1) + \ldots + E(k,c^{Y^{f(Y, c)}},2^\ell-1),
\]
where again $Y=k^{2^\ell}$. This gives a bound on
$E(k, c, 2^{\ell+1}-1).$
\end{proof}

\section{The full result}\label{se:full}

In this section, we give the full proof of Theorem~\ref{th:hilbert},
first for any number of colors, but $k=2$, and then for any $k$ as
well. As before, we view pairs of integers as ordered pairs $(x,y)$
with $x < y$. When we have a grid $\{(x+id,y+jd)\}$ for a range of
values $i$ and $j$, we will say the grid is in position $(x,y)$ with
scale $d$.

\subsection{Any colors, two dimensions}\label{se:allcolors}

\begin{proof}[Proof of Theorem~\ref{th:hilbert} when $k=2$]
As in the proof of Lemma~\ref{le:embedding}, we first give the arguments
ignoring the numbers involved, and in the next section we determine a
bound on $n(r,2)$.

Begin with an $r$-coloring $\chi_0 = \chi$ of a large initial grid,
$G_{\lambda}$. By Gallai-Witt, find a large monochromatic subgrid above the
diagonal $x=y$, with color $c_\lambda$ in position $(x_0, y_0)$
with scale $d_0$.

As in the proof with two colors, this yields two grids, $G_1$ and $G_2$
of equal size, in positions $(x_0, x_0)$ and $(y_0, y_0)$ respectively,
both with scale $d_0$. Note that these grids contain points on, above,
and below the diagonal $x=y$ --- we only consider those points above
the diagonal. As in the proof in Section~\ref{se:twocolors}, if two
points in these grids of the form $(x_0+id, x_0+jd)$ and $(y_0+id, y_0+jd)$
are both the same color as the grid $G_{\lambda}$, then we get our
monochromatic Hilbert cube of dimension 2.
The colorings of $G_1$ and $G_2$ correspond to $\chi_A$ and $\chi_B$ from
the initial proof. We consider a $r^2$-coloring of a new grid, where the point
$(i,j)$ is colored by the pair
\[
\chi_1(i,j) = (\chi_0(x_0 + id, x_0 + jd), \chi_0(y_0 + id, y_0 + jd)).
\]

We now use Gallai-Witt with $r^2$ colors, to find a large subgrid under
$\chi_1$ with color $(c_1, c_2)$ in position $(x_1, y_1)$ with scale $d_1$.
This grid really corresponds to two grids: one of color $c_1$ in position
$(x_0 + x_1 d_0, x_0 + y_1 d_0)$, and the other of color $c_2$ in position
$(y_0 + x_1 d_0, y_0 + y_1 d_0)$. Both grids have scale $d_0 d_1$, and they
are entirely contained in grids $G_1$ and $G_2$ respectively.

Again we pass to subgrids. The grid in $G_1$ yields two subgrids $G_{11}$
and $G_{12}$, in positions $(x_0 + x_1 d_0, x_0 + x_1 d_0)$ and
$(x_0 + y_1 d_0, x_0 + y_1 d_0)$ respectively, both with scale $d_0 d_1$.
Likewise $G_2$ give us two subgrids, $G_{21}$, and $G_{22}$. Now we
have more ways to win: the colorings of $G_{11}$ and $G_{12}$ restrict
each other, as do $G_{21}$ and $G_{22}$, and both of $G_{11}, G_{12}$
restrict both of $G_{21}, G_{22}$. Note that, whether the position of
the grid involves $x_0$ or $y_0$ is determined by the first part of
the subscript, and whether it involves $x_1$ or $y_1$ is dependent
on the next part.

\begin{figure}[t]
  \centering
  \includegraphics[scale=0.5]{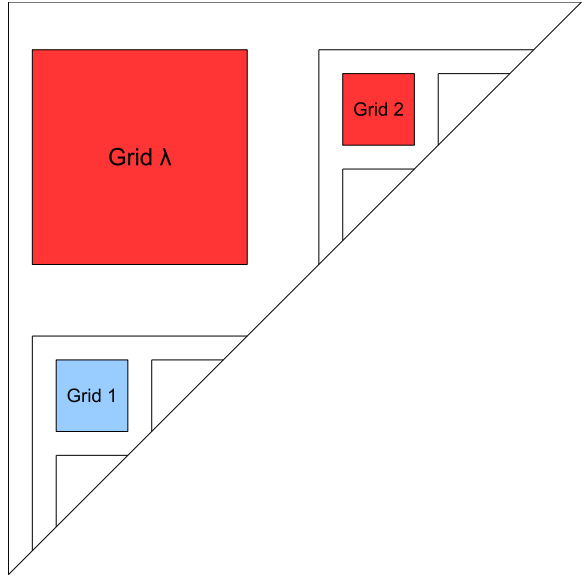}
  \caption{The sequence of subgrids}
  \label{fig:deepgrid}
\end{figure}

The next step, which we briefly state, is to define a grid-coloring $\chi_2$
with $r^4$ colors corresponding to each of the four grids
$G_{11}, G_{12}, G_{21}, G_{22}$. We find a subgrid of color
$(c_{11}, c_{12}, c_{21}, c_{22})$ under this coloring, which corresponds
to four grids, which further restrict one another.

Continue this for $f(2,r)+1$ steps, so that the final grids are indexed
by strings of length $f(2,r)$. The ``large'' monochromatic grid we find
under $\chi_{f(2,r)-1}$ need only be a $2 \times 2$ grid, giving $G_s$
a single off-diagonal point for all $s$ of length $f(2,r)$. The color of
this point is $c_s$.

We now recognize the map $s \mapsto c_s$ as an $r$-coloring of the perfect
binary tree of height $f(2,r)$. By the definition of $f$, this coloring must
be 1-balanced, meaning there is a node $\sigma$ and
two children $s = \sigma \cdot 1 \cdot u$ and $t = \sigma \cdot 2 \cdot v$,
all the same color, where $u,v \in \{1,2\}^\ell$ for some $\ell$.
Call this color red.

Write $\sigma = \sigma_0 \sigma_1 \ldots \sigma_{k-1}$. Since $\sigma$ is red,
the monochromatic grid found in grid $G_{\sigma_k}$
is red. Let
\[
z_i(\sigma) = \left\{ \begin{array}{ll}
x_i & \hbox{if } \sigma_i = 1 \\
y_i & \hbox{if } \sigma_i = 2.\\
\end{array} \right.
\]
Then the grid $G_\sigma$ is in position $(X(\sigma),Y(\sigma))$, where
\[
\begin{array}{rcl}
X(\sigma) & = & z_0(\sigma) + d_0 (z_1(\sigma) + d_1( \ldots (z_{k-1}(\sigma) + d_{k-1}x_k) \ldots)), \\
Y(\sigma) & = & z_0(\sigma) + d_0 (z_1(\sigma) + d_1( \ldots (z_{k-1}(\sigma) + d_{k-1}y_k) \ldots)). \\
\end{array}
\]
and has scale $D = d_0 d_1 \cdots d_k$. Note that the only difference
between $X$ and $Y$ is the $x_k$ and $y_k$ respectively in the inner-most
term.

Now we look at the grids $G_s$ and $G_t$. We will only use a single point
from these grids. Define $z_i, X$, and $Y$ in the same way as above for
$s$ and $t$. Noting that
\[
s_0=\sigma_0, s_1=\sigma_1, \ldots, s_{k-1}=\sigma_{k-1}, s_k=1 \hbox{ and}
\]
\[
t_0=\sigma_0, t_1=\sigma_1, \ldots, t_{k-1}=\sigma_{k-1}, t_k=2,
\]
we see that $G_s$ is in position $(X(s),Y(s))$ with
\[
\begin{array}{rcl}
X(s) & = & X(\sigma) + D (x_k + d_k (\ldots (z_{k+\ell-1}(s) + d_{k+\ell-1} x_{k+\ell}) \ldots), \\
Y(s) & = & X(\sigma) + D (x_k + d_k (\ldots (z_{k+\ell-1}(s) + d_{k+\ell-1} y_{k+\ell}) \ldots), \\
\end{array}
\]
and similarly $G_t$ is in position $(X(t),Y(t))$ with
\[
\begin{array}{rcl}
X(t) & = & Y(\sigma) + D (y_k + d_k (\ldots (z_{k+\ell-1}(t) + d_{k+\ell-1} x_{k+\ell}) \ldots), \\
Y(t) & = & Y(\sigma) + D (y_k + d_k (\ldots (z_{k+\ell-1}(t) + d_{k+\ell-1} y_{k+\ell}) \ldots). \\
\end{array}
\]

We claim that $X(s), X(t), Y(s), Y(t)$ form our Hilbert cube.
Indeed, writing $a=X(s)$, $b=X(t)-X(s)=Y(t)-Y(s)$, and
\[
c = D d_k \cdots d_{k+\ell} (y_{k+\ell+1} - x_{k+\ell+1}),
\]
we see that they have the form $a, a+b, a+c, a+b+c$ respectively.

Now consider the colors of the six points among these values (still
only looking at points above the line $x=y$). Since the points
$(X(s),Y(s))$ and $(X(t),Y(t))$ are in $G_s$ and $G_t$ respectively,
we know that both points are red.

Now we recognize that these values are given by
\[
\begin{array}{rcl}
X(s) & = & X(\sigma) + iD, \\
Y(s) & = & X(\sigma) + jD, \\
X(t) & = & Y(\sigma) + iD, \\
Y(t) & = & Y(\sigma) + jD, \\
\end{array}
\]
so the four points we need look like
\[
\begin{array}{rcl}
(X(s),X(t)) & = & (X(\sigma) + iD, Y(\sigma) + iD) \\
(X(s),Y(t)) & = & (X(\sigma) + iD, Y(\sigma) + jD) \\
(Y(s),X(t)) & = & (X(\sigma) + jD, Y(\sigma) + iD) \\
(Y(s),Y(t)) & = & (X(\sigma) + jD, Y(\sigma) + jD).\\
\end{array}
\]
By design, these fall into the grid $G_\sigma$, so these points
are red as well.
\end{proof}

\subsection{Upper bounds}\label{se:bounds}
The process repeats to a depth of $f(2,r)$, at which point we have
$2^{f(2,r)}$ grids, meaning $r^{2^{f(2,r)}}$ colors. At this level, we are
looking for a square, so these grids must have size
\[
S_{f(2,r)} = 2.
\]
At the prior level, our $2^{f(2,r)-1}$ grids must have monochromatic subgrids
of size $S_{f(2,r)}$, and the joint coloring has $r^{2^{f(2,r)-1}}$ colors.
Thus
\[
S_{f(2,r)-1} = 2 GW(S_{f(2,r)}, r^{2^{f(2,r)-1}}),
\]
where the factor of 2 allows us to take the top-left quadrant of the grid.
As before, this ensures distinct values in the $x$ and $y$ components.
Repeating this reasoning, we find that
\[
S_{k} = 2 GW(S_{k+1}, r^{2^{k}}),
\]
which leaves us with this bound for the size of the initial grid:
\[
n(r,2) \le S_{0} = 2 GW(S_1, r).
\]

\subsection{Any colors, any dimensions}
We have now done all of the hard work. In order to prove the full result
at this point, we only need to reconsider the proof for $k=2$.

\begin{repeatedtheorem}{\ref{th:hilbert}}
For all $r, k$, there is a number $n = n(r, k)$ so that for any
$r$-coloring of the edges of the complete graph on $[n]$, there
is a Hilbert cube $H = H(a; b_1, \ldots, b_k)$ so that all edges within
$H$ are monochromatic.
\end{repeatedtheorem}

\vspace{1pc}
\begin{proof}
Let $\chi$ be an $r$-coloring of a large grid.
Repeat the process from the proof in Section~\ref{se:allcolors},
now continuing until we have a tree of height $E(2,r,k-1)$.

By Lemma~\ref{le:embedding}, there is an embedded tree of height $k-1$
which is entirely, say, red. Call the embedding $\varphi$, so the nodes
are labeled $\varphi(s)$ for $s \in \{1,2\}^j$ for $0 \le j < k$.

Let $G_s$ denote the red grid corresponding to the node
$\varphi(s)$.\footnote{In the previous proof, we would have called
this $G_{\varphi(s)}$, but here we have no need to refer to the nodes outside
of our monochromatic tree.}
Say this grid is in position $(X(s), Y(s))$. If $i$ is the length of $s$,
then the scale of $G_s$ is $d_0 d_1 \cdots d_i$.

For each $s \in \{1,2\}^{k-1}$, consider the red point
$(X(s),Y(s)) \in G_s$. We claim that the $2^k$ values
\[
\{ X(s) \mid s \in \{1,2\}^{k-1} \} \cup \{ Y(s) \mid s \in \{1,2\}^{k-1} \}
\]
have the form $a + \sum_{i \in I} b_i$ and comprise an entirely red clique.

As we saw in the previous proof, for $s$ on level $\ell-1$,
and $s \cdot 1, s \cdot 2$ on level $\ell$,
\[
X(s \cdot 2) - X(s \cdot 1) = Y(s) - X(s)
  = d_0 d_1 \cdots d_{\ell-1} (y_{\ell} - x_{\ell}).
\]
Inspired by this, we define
\[
b_{\ell} = Y(s) - X(s)
\]
for $s$ on level $\ell-1$.

Now set $a = X(1^{k-1})$. Let $s = s_1 \cdots s_{k-1} \in \{1,2\}^{k-1}$.
Let $I = \{i \mid s_i = 2\} \subseteq [k-1]$. This gives us
$X(s) = a + \sum_{i \in I} b_i$ and $Y(s) = a + b_k + \sum_{i \in I} b_i$.

This tells us the numbers we are looking at really do have the desired form.
We only need to check that all the edges among these values are red.

Let $s$ be any string on level $k-1$. By virtue of $(X(s),Y(s))$ being
a point in the grid $G_s$, we know that edge is red. Now let $t$ be another
string on level $k-1$, and assume $s < t$ lexicographically. Let $\sigma$
be the longest initial string that $s$ and $t$ agree on --- their closest
common ancestor. Since $s < t$, we must have that $s = \sigma \cdot 1 \cdot u$
and $t = \sigma \cdot 2 \cdot v$ for some $u$ and $v$ of the same length.

As we saw in the previous proof, since $G_{\sigma}$ is red, we immediately
get that $(X(s),X(t)), (X(s),Y(t)), (Y(s),X(t)), (Y(s),Y(t))$ are all red.

By considering all possible $s,t$ on level $k-1$, this argument says that
all edges among these values are red, so we have reached our goal.
\end{proof}

Along the same lines as Section~\ref{se:bounds}, we may define the recurrence
\[
T_{E(2,r,k)} = 2, \hbox{ and}
\]
\[
T_{k} = 2 GW(T_{k+1}, r^{2^{k}}),
\]
to get an upper bound of
\[
n(r,k) \le T_{0} = 2 GW(T_1, r).
\]

\subsection{Additional results}

Theorem~\ref{th:hilbert} immediately gives several nice consequences.

By considering subsets of Hilbert cubes, it is easy to see that, for
large $n$, any edge-coloring of the complete graph on $[n]$ will always
have solutions to equations of the form
$x_1 + \ldots + x_\ell = y_1 + \ldots + y_\ell$
which induce monochromatic subgraphs.

Combining Theorem~\ref{th:hilbert} with Szemer\'edi's celebrated
theorem on arithmetic progressions \cite{szemeredi}, we get the
following nice corollary.

\begin{corollary}
For any $\delta > 0$, and naturals $r,k$, there is a number
$n = n(r,\delta)$ so that for any set $A \subseteq \N$ of upper
density $\delta$, and any $r$-coloring of the edges of the complete
graph on $A$, there is a Hilbert cube $H = H(a; d_1, \ldots, d_k)$
contained in $A$ so that all edges within $H$ are monochromatic.
\end{corollary}

On the other hand, our theorem also inspires another negative result.
A Hilbert cube of dimension 2 is simply a set satisfying $w-x=y-z$.
We consider a similar equation, $a(w-x)=b(y-z)$, for $a \neq b$ fixed.
To avoid this equation, color pairs based on their difference.
Write $|w-x| = \left( \frac{b}{a} \right)^k p$ for $k$ as large as
possible, and color $\{w,x\}$ by the parity of $k$. Since $(w-x)$ and
$(y-z)$ will always be different by a factor of $\frac{b}{a}$, this
will assure the edges $\{w,x\}$ and $\{y,z\}$ have different colors.

\section{Acknowledgment}
The author would like to thank Ron Graham for his guidance, which
helped to strengthen the main theorem.

\end{document}